\documentclass{amsart}

\usepackage{eufrak}

\usepackage{amssymb,amsmath, amscd, mathrsfs, yfonts, bbm, graphics}
\title[Freeness with amalgamation, limit theorems, $S$-transform, type B freeness]
{Freeness with amalgamation, limit theorems and $S$-transform in
Non-commutative probability spaces of type B}

\author{Mihai Popa}
\address{Indiana University at Bloomington,
Department of Mathematics, Rawles Hall, 931 E 3rd St, Bloomington,
IN 47405} \email{mipopa@indiana.edu}
\address{Institute of Mathematics ''Simion Stoilow'' of the Romanian
Academy, P.O. Box 1-764 RO-014700 Bucharest, Romania}
\DeclareMathAlphabet{\mathpzc}{OT1}{pzc}{m}{it}
\newtheorem{claim}{}[section]
\newtheorem{defn}[claim]{Definition}
\newtheorem{thm}[claim]{Theorem}

\newtheorem{remark}[claim]{Remark}

\newtheorem{cor}[claim]{Corollary}
\newtheorem{conseq}[claim]{Consequence}

\newcommand{\cA}{\mathcal{A}}

\newcommand{\lra}{\longrightarrow}

\newcommand{\C}{\mathcal{C}}

\newcommand{\n}{\mathcal{X}}

\newcommand{\sys}{(\cA,\varphi,\n,f, \Phi)}

\newcommand{\ka}{\kappa}
\newcommand{\seco}{\prime\prime}
\newcommand{\bcnv}{\framebox(4,6){$\star$}}
\newcommand{\E}{\widetilde{E}}
\newcommand{\G}{\mathcal{G}}
\newcommand{\pbcnv}{\check{\bcnv}}

\usepackage{amssymb}
\input xy
\xyoption{all}

\begin{document}

 \maketitle
\bibliographystyle{alpha}

\begin{abstract}The present material addresses several problems left
 open in the Trans. AMS paper " Non-crossing cumulants of type B" of P.
 Biane, F. Goodman and A. Nica. The main result is that
 a type B non-commutative probability space can be studied in the
 framework of freeness with amalgamation. This view allows easy ways of
 constructing a version of the S-transform as well as proving analogue
 results to Central Limit Theorem and Poisson Limit Theorem.
\end{abstract}

\section{introduction}

The present material addresses several problems left open in the
paper ''Non-crossing cumulants of type B'' of P. Biane, F. Goodman
and A. Nica (reference \cite{bgn}).

The type $A,B,C$ and $D$ root systems determine correspondent
lattices of non-crossing partitions (see \cite{reiner},
\cite{athanasiadis}). The type $A_{n+1}$ corresponds to the lattice
of non-crossing partitions on the ordered set $[n]=1<\cdots<n$; the
types $B_n$ and $C_n$ determine the same lattice of non-crossing
partitions on $[\overline{n}]=1<\cdots<n<-1<\cdots<-n$, namely the
partitions with the property that if $V$ is a block, then $-V$ (the
set containing the opposites of the elements from $V$) is also a
block; the type $D$ corresponds to a lattice of the symmetric
non-crossing partitions with the property that if there exists a
symmetric block, then it has more than 2 elements and contains $-n$
and $n$. (see again \cite{reiner}, \cite{athanasiadis},
\cite{athrein}).

The lattices of type $A$ and type $B$ non-crossing partitions are
self-dual with respect to the Kreweras complementary. In the type
$A$ case, the lattice structure was known to be connected the
combinatorics of Free Probability Theory (see \cite{nisp}). For the
type $B$ case, the properties of the lattice allow also a
construction, described in \cite{bgn}, of some associated
non-commutative probability spaces, with a similar apparatus as in
the type $A$ case (such as $R$-transform and boxed convolution). The
paper \cite{bgn} leaves open some questions on these objects:
possible connections to other types of independence, limit theorems,
$S$-transform. The main observation of the present material is that
a type $B$ non-commutative probability space can be studied in the
framework of freeness with amalgamation, that gives fast answers to
the rest of the problems.

The material is organized as follows: second section reviews some
results from \cite{bgn}; third section presents the connection with
freeness with amalgamation; forth section is briefing the
construction of the $S$-transform for the type B non-commutative
probability spaces, utilizing the commutativity of the matrix
algebra $C$; fifth and, respectively, sixth section are presenting
limit results: analogues of central limit theorem, respectively
Poisson limit theorem.

\section{preliminary results}

\begin{defn}\emph{A non-commutative probability space of type B is a
system\\
 $\sys$, where:}
\begin{enumerate}
\item[(i)]$(\cA,\varphi)$ \emph{is a non-commutative probability space (of
type A), i.e. $\cA$ is a complex unital algebra and
$\varphi:\cA\lra\mathbb{C}$ is a linear functional such that}
$\varphi(1)=1$.
\item[(ii)]$\n$ \emph{is a complex vector space and $f:\n\lra\mathbb{C}$
is a linear functional.}
\item[(iii)]$\Phi:\cA\times\n\times\cA\lra\cA$ \emph{is a two-sided
action of $\cA$ on $\nu$ (when there is no confusion, it will be
written ''$a\xi b$ instead of} $\Phi(a\xi b)$, for $a,b\in\cA$ and
$\xi\in\nu$)
\end{enumerate}
\end{defn}

On the vector space $\cA\times\n$ it was defined a structure of
unital algebra considering the multiplication:
\[(a,\xi)\cdot (b,\eta)=(ab, a\eta + \xi b), \ a,b\in\cA,\ \xi,\eta\in\n\]

The above algebra structure can be obtained when
$(a,\xi)\in\cA\times\n$ is identifies with a $2\times2$ matrix,
\[(a,\xi)\leftrightarrow\left[\begin{array}{cc}a& \xi\\ 0& a\\ \end{array}\right].\]
We will consider also the commutative unital algebra $\C$ by
similarly endowing the vector space $\mathbb{C}\times\mathbb{C}$
with the multiplication:
\[(x,t)\cdot (y,s)=(xy, xs+ty),\]
i.e.  using the identification
\[\C\ni(x,t)\leftrightarrow\left[\begin{array}{cc}x& t\\ 0& x\\ \end{array}\right]
\in M_{2}(\mathbb{C}).\]

\begin{defn}
\emph{Let $\sys$ be a non-commutative probability space of type B.
The non-crossing cumulant functionals of type B are the families of
multilinear functionals}  $\left(
\ka_n:(\cA\times\n)^n\lra\C\right)_{n=1}^\infty$ \emph{defined by
the following equations: for every $n\geq1$ and every $a_1,\dots,
a_n\in\cA,\xi_1,\dots,\xi_n\in\n$, we have that:}
\begin{equation}\label{mcum}
\sum_{\gamma\in NC^{(A)}(n)}\prod_{B\in\gamma}\ka_{card(B)}
\left((a_1,\xi_1)\cdots(a_n,\xi_n)|B\right)=E\left((a_1,\xi_1)\cdots(a_n,\xi_n)\right)
\end{equation}
\emph{where the product on the left-hand side is considered with
respect to the multiplication on $\C$ and the product
$(a_1,\xi_1)\cdots(a_n,\xi_n)$ on the right-hand side is considered
with respect to the multiplication on $\cA\times\n$ defined above.}
\end{defn}
Note that the first component of $\ka_{m}
\left((a_1,\xi_1)\cdots(a_n,\xi_m)\right)$ equals the non-crossing
cumulant $k_m(a_1,\dots,a_m)$.

We will also use the notation $\ka_n(a,\xi)$ for $\ka_{n}
\left((a,\xi)\cdots(a,\xi)\right)$ and $M_n$ for
$E\left((a,\xi)^n\right)$.

\begin{defn}
\emph{Let $\cA_1,\dots,\cA_k$ be unital subalgebras of $\cA$ and let
$\n_1,\dots,\n_k$ be linear subspaces of $\n$ such that each $\n_j$
is invariant under the action of $\cA_j$. We say that
$(\cA_1,\n_1),\dots (\cA_k,\n_k)$ are free independent if}
\[\kappa_n\left((a_1,\xi_1),\dots,(a_n,\xi_n)\right)=0\]
\emph{whenever} $a_l\in\cA_{i_l}, \xi_l\in\n_{i_l}\ (l=1,\dots,n)$
 \emph{are such that there exist} $1\leq s<t\leq n$ \emph{with} $i_s\neq
 i_t$.
\end{defn}

For $(a,\xi)\in\cA\times\n$ we consider the moment and cumulat
 or $R$-transform, series:
\begin{eqnarray*}
M(a,\xi)&=&\sum_{n=1}^\infty \left(E\left((a,\xi)^n\right)\right)z^n\\
R(a,\xi)&=&\sum_{n=1}^\infty\ka_n(a,\xi)z^n
\end{eqnarray*}
\begin{defn}
\emph{Let $\Theta^{(B)}$ be the set of power series of the form:
\[f(z)=\sum_{n=1}^\infty (\alpha^\prime_n, \alpha^{\seco}_n)z^n,\]
where $\alpha^\prime_n,\alpha^{\seco}_n$ are complex numbers. For
$p\in NC^{(A)}(n)$ and $f\in\Theta^{(B)}$, consider
\[Cf_p(f)=\prod_{B\in p}(\alpha^\prime_{|B|},\alpha^{\seco}_{|B|})\]
(the right-hand side product is in $\C$.)}
 \emph{On $\Theta^{(B)}$ we
define the binary operation $\bcnv$ by:}
\begin{eqnarray*}
f\bcnv g &=&\sum_{n=1}^\infty (\gamma^\prime_n,\gamma^{\seco}_n)z^n\ \text{where}\\
(\gamma^\prime_n,\gamma_n^{\seco})&=&\sum_{p\in
NC^{(A)}(n)}Cf_p(f)Cf_{Kr(p)}(g)
\end{eqnarray*}
\end{defn}

\begin{thm}\emph{The moment series $M$ and $R$-transform $R$ of $(a,\xi)$
are related by the formula}
\[M=R\bcnv \zeta^\prime\]
\emph{where $\zeta^\prime\in\Theta^{(B)}$ is the series
$\sum_{n=1}^\infty (1,0)z^n$.}
\end{thm}

\begin{remark}\label{components}\emph{We denote by $k^\prime_{n,p}$ or, for simplicity,
by $k^\prime_{n}$, the multilinear functional from $\cA^{p-1}\times
\n\times\cA^{n-p}$ to $\mathbb{C}$ which is defined by the same
formula as for the (type A) free cumulants
$k^n:\cA^n\lra\mathbb{C}$, but where the $p$th argument is a vector
from $\n$ and $\varphi$ is replaced by $f$ in all the appropriate
places. The connexion between the type B cumulants $\kappa_n$ and
the functionals $k_n, k_n^\prime$ is given by:}
\begin{equation}\label{abconnex}
\kappa_n((a_1,\xi_1),\dots,
(a_n,\xi_n))=\left(k_n(a_1,\dots,a_n),
\sum_{p=1}^nk^\prime_n(a_1,\dots,a_{p-1},\xi_p,a_{p+1},\dots,a_n)
\right)
\end{equation}

\end{remark}

\begin{thm}\label{freethm}\emph{If $(\cA_1,\n_1),(\cA_2,\n_2)$ are free independent,
$(a_1,\xi_1)\in(\cA_1,\n_1),(a_2,\xi_2)\in(\cA_2,\n_2)$, and $R_1$,
respectively $R_2$ denote the $R$-transforms of $(a_1,\xi_1)$ and
$(a_2,\xi_2)$, then:}
\begin{enumerate}
\item[(i)] \emph{the $R$-transform of $(a_1,\xi_1)+(a_2,\xi_2)$ is
$R_1+R_2$.}
\item[(2)]\emph{the $R$-transform of $(a_1,\xi_1)\cdot(a_2,\xi_2)$ is
$R_1\bcnv R_2$.}
\end{enumerate}
\end{thm}

\section{connexion to "freeness with amalgamation"}

As shown in \cite{bgn}, Section 6.3, Remark 3, the definitions of
the type B cumulants are close to those from the framework of the
"operator-valued cumulats", yet some detailes are different - mainly
the map $E$ is not a conditional expectation and $\cA\times\n$ is
not a bimodule over $C$. Following a suggestion of Dimitri
Shlyakhtenko, the construction of the type B probability spaces can
still be modified in order to overcome these points.

Let $\textgoth{E}=\n\oplus\cA$. On $\cA\times\textgoth{E}$ we have a
$C$-bimodule structure given by:

\[(x,t)(a,\xi+b)=(a,\xi+b)(x,t)=(ax, at+(\xi+b)x)\]
for any $x,t\in\mathbb{C}, a,b\in\cA, \xi\in\n$. Since $\cA$ is
unital, $C$ is a subspace of $\textgoth{E}$.

The map $E$ extends to $\textgoth{E}$ via:
\[\E(a,\xi+b)=\left(\varphi(a),f(\xi)+\varphi(b)\right)\]
The extension becomes a conditional expectation, since:
\begin{eqnarray*}
\E\left((x,t)(a,\xi+b)\right)&=&\E(ax, at+(\xi+b)x)\\
&=&\left(\varphi(ax), \varphi(ta)+f(\xi x)+\varphi(bx)\right)\\
&=&\left(x\varphi(a), t\varphi(a)+x f(\xi)+x\varphi(b)\right)\\
&=&(x,t)\left(\varphi(a),f(\xi)+\varphi(b)\right)\\
&=&(x,t)\E(a,\xi+b)
\end{eqnarray*}

The equation \ref{mcum} can naturally be extended in the framework
of $\textgoth{E}$ and $\E$, framework that reduces the construction
to freeness with amalgamation, namely defining the cumulants
$\widetilde{\kappa}$ by the equation:
\begin{equation}\label{mcumw}
\sum_{\gamma\in
NC^{(A)}(n)}\prod_{B\in\gamma}\widetilde{\ka}_{card(B)}
\left((a_1,\xi_1)\cdots(a_n,\xi_n)|B\right)=\widetilde{E}\left((a_1,\xi_1)\cdots(a_n,\xi_n)\right)
\end{equation}

 If ${m}:\cA\times\cA\ni (a,b)\mapsto m(a,b)=ab\in\cA$
 is the multiplication in $\cA$, note that
 $(\cA,\varphi,\n\oplus\cA,f\oplus\varphi, \Phi\oplus{m})$ is
 also a type B noncommutative probability space, therefore
 Remark \ref{components} (i.e. Theorem 6.4 from
 \cite{bgn}) gives the components of $\widetilde{\ka}$:
 \begin{eqnarray*}
\widetilde{\kappa}_n((a_1,\xi_1+b_1),\dots, (a_n,\xi_n+b_n)) &=&\\
&&\hspace{-2.7cm} \left(k_n(a_1,\dots,a_n),
\sum_{p=1}^nk^\prime_n(a_1,\dots,a_{p-1},\xi_p+b_p,a_{p+1},\dots,a_n)
\right)\\
 \end{eqnarray*}

\section{the $S$-transform}
Utilizing the commutativity of the algebra $C$, the construction of
the $S$-transform is essentially a verbatim reproduction of the type
A situation.

We will denote
\[\G=\{\sum_{n=1}^\infty\alpha_nz^n, \alpha_n\in C\}\]
the set of formal series without constant term with coefficients in
$C$, and
\[\G^{\langle -1\rangle}=\{\sum_{n=1}^\infty\alpha_nz^n, \alpha_n\in C, \alpha_1=\text{invertible}\}\]
the set of all invertible series (with respect to substitutional
composition) with coefficients in $C$ ( see \cite{aep}).
\begin{defn}\emph{Let $(a,\xi)\in\cA\times\n$ such that
$\varphi(a)\neq 0$, that is $(\varphi(a),f(\xi))$ is invertible in
$C$. If $R_{(a,\xi)}(z)$ is the $R$-transform series of $(a,\xi)$,
then the \emph{$S$-transform} of $(a,\xi)$ is the series defined by}
\[S_{(a,\xi)}(z)=\frac{1}{z}R^{\langle -1\rangle}_{(a,\xi)}(z)\]
\end{defn}
\begin{thm}\emph{If $\cA_1,\n_1),(\cA_2,\n_2)\subset(\cA,\n)$ are free
independent and $(x_j,\xi_j)\in(\cA_j,\n_j), j=1,2$ are such that
$\varphi(x_j)\neq0$, then:}
\[S_{(a_1,\xi_1)(x_2,\xi_2)}(z)=S_{(a_1,\xi_1)}(z)S_{(a_2,\xi_2)}(z)\]

\end{thm}

\begin{proof}
The proof presented in \cite{nisp}, for the type $A$ case, works
also for the freeness with amalgamation over a commutative algebra.
Yet, for the convenience of the reader, we will outline the main
steps.

Since, for $(a_1,\xi_1), (a_2,\xi_2)$ free,
$R_{(a_1,\xi_1)\cdot(a_2,\xi_2)}=R_{(a_1,\xi_1)}\bcnv
R_{(a_2,\xi_2)}$, it suffices to prove that the mapping
\[\mathcal{F}:\G^{\langle -1\rangle}\ni
 f\mapsto \frac{1}{z}f^{\langle -1\rangle}\in\G\]
has the property
\begin{equation}\label{unu}
\mathcal{F}\left(f\bcnv g\right)=\mathcal{F}(f)\mathcal{F}(g).
\end{equation}

Indeed, (\ref{unu}) is equivalent to

\begin{equation}\label{doi}
z\left(f\bcnv g\right)=f^{\langle -1\rangle}\left(f\bcnv
g\right)\cdot g^{\langle -1\rangle}\left(g\bcnv f\right)
\end{equation}

For $\sigma\in NC(n)$ and  $h=\sum_{n\geq 1} h_nz^n$, we define
\[Cf_\sigma(h)=\prod_{B\in\sigma} h_{card(B)}\in C.\]
Also, for $f,g\in\G$, we denote
\[\left(f\pbcnv g\right)(z)=\sum_{n\geq1} \lambda_n z^n\]
where ($K(\sigma)$ is the Kreweras complementary of $\sigma$)
\[\lambda_n=\sum_{
 \begin{array}{ll}
\sigma\in NC(n)\\
 (1)\ \text{block in} \ \sigma\\
\end{array}
} Cf_\sigma(f)\cdot Cf_{K(\sigma)}(g)
\]

For $f=\sum_{n\geq1}\alpha_nz^n\in\G^{\langle-1\rangle}$ we have
that:
\[f^{\langle-1\rangle}\circ\left(f\bcnv g\right)
=\alpha_1^{-1}\left(f \pbcnv g\right)\] since, with the above
notations, the coefficient of $z^m$ in the right hand side is
\[\sum_{n\geq1}\sum_{
\begin{array}{cc}
i_1,\dots,i_n\geq1\\
 i_1+\dots+i_n=m\\
\end{array}
}
\alpha_n\alpha_1^{-n}\lambda_{i_1}\cdots\lambda_{i_n}
\]
while the coefficient of $z^m$ in the left-hand side is
\[
\sum_{n\geq1} \sum_{1=b_1<\dots b_n\leq m} \sum _{
\begin{array}{cc}
\pi\in NC(m)\\
(b_1,\dots,b_n)\in\pi\\
\end{array}
}
 Cf_\pi(f)\cdot
Cf_{K(\pi)}(g)
\]
and the equality follows setting $\pi_k=\pi|\{b_k,\dots,b_{k+1}-1\}$
(notationally $b_{n+1}=m$) and remarking that $K(\pi)$ is the
juxtaposition of $K(\pi_1),\dots,K(\pi_n)$.

It follows that , if $\{\alpha_n\}_{n\geq1}, \{\beta_n\}_{n\geq1}$
are respectively the coefficients of $f$ and $g$,(\ref{doi}) is
equivalent to
\[
\left(f\pbcnv g\right)(z)\cdot \left(f\pbcnv
g\right)(z)=\alpha_1\beta_1 z\cdot \left(f\bcnv g\right)(z)
\]
The coefficient of $z^{m+1}$ on the left-hand side is
\[
\sum_{n=1}^m
\sum_{
\begin{array}{cc}
\pi\in NC(n)\\
(1)\in\pi\\
\end{array}
}
 \sum_{
\begin{array}{cc}
\rho\in NC(m+n-1)\\
(1)\in\rho\\
\end{array}
} Cf_\pi(f)\cdot Cf_{K(\pi)}(g) \cdot Cf_\rho(g)\cdot
Cf_{K(\rho)}(f)
\]

while the coefficient  of $z^{m+1}$ on the right-hand side is
\[
\sum_{\sigma\in NC(m)}\alpha_1\beta_1
\cdot Cf_{\sigma}(f)\cdot Cf_{K(\sigma)}(g).
\]

 As shown in \cite{nisp}, the conclusion follows from the bijection
 between the index sets of the above sums. More precisely, if $1\leq n \leq m$, to the
 pair consisting on
 $\pi\in NC(n)$ and $\rho\in NC(m+1-n)$ both contain the block $(1)$,
 we associate the partition from $NC(n+m-1)$ obtained by juxtaposing
 $\pi\setminus (1)$ and $K(\rho)$.
\end{proof}

\section{central limit theorem}
\begin{thm}\emph{Let $\{(\cA_k,\n_k)\}_{k\geq 1}\subset(\cA,\n)$ be type $B$ free
independent and $(x_k,\xi_k)\in(\cA_k,\n_k)$ identically distributed
such that $\varphi(x_k)=f(\xi_k)=0$ and
$\varphi(x_k^2)=f(\xi_k^2)=1$. The limit distribution moments of
\[\frac{(a_1,\xi_1)+\dots+(a_N,\xi_N)}{\sqrt{N}}\]
are $\{m_n, \mathfrak{m}_n\}_n$, where $\{m_n\}_n$ are the moments
of the semicircular distribution and}

\[\mathfrak{m}_n=\left\{\begin{array}{cc}
0&\text{if $n$ is odd}\\ {2k \choose {k+1}} &
 \text{if $n=2k$ is even}.\end{array}\right.\]
\end{thm}
 \begin{proof}
 Note $S_N=\frac{(a_1,\xi_1)+\dots+(a_N,\xi_N)}{\sqrt{N}}$ and
 $R_N=R(S_N)$. Theorem \ref{freethm} implies
 \[\lim_{N\rightarrow\infty}R_N=(1,1)z^2\]
 The first component of the limit distribution is the
 Voiculescu's semicircular distribution. To compute the second
 component of the moments,
 we will use the equation (\ref{mcum}), which
 becomes:
\begin{eqnarray*}
E\left((a_1,\xi_1)^n\right)&=& \sum_{\gamma\in NC^{(A)}_2(n)}\ka_{2}
\left((a_1,\xi_1)\right)^{\frac{n}{2}}\\
\end{eqnarray*}
It follows that all the odd moments are zero, and, since in $C$, $
(a,b)^n=(a^n, na^{n-1}b)$, the even moments are given by:
\begin{eqnarray*}
\mathfrak{m}_{2n}&=&nC_n,\ \text{where $C_n$ stands for the $n$-th
Catalan number}\\
&=&n\frac{1}{n+1}{2n \choose{n}}\\
&=&{2n \choose {n+1}}.
\end{eqnarray*}

 \end{proof}

 \begin{remark} \emph{The second components of the above limit moments are not
 the moments of positive Borel measure on $\mathbb{R}$. Yet, they are
 connected to the moments of another remarkable distribution
 appearing in non-commutative probability - the central limit
 distribution for monotonic independence.}

  \emph{For variables that are monotonically independent
 (see \cite{muraki1}, \cite{muraki2}), the
 limit moments in the Central Limit Theorem are given by the "arsine
 law", i.e. the $n$-th moment $\mu_n$ is given by}
 \[\mu_n=\left\{\begin{array}{cc}
0&\text{if $n$ is odd}\\ {2k \choose {k}}=(k+1)C_k &
 \text{if $n=2k$ is even.}\end{array}\right.\]

 \emph{Hence $\mu_n=m_n+\mathfrak{m}_n$, which implies the
 following:}
 \begin{cor}\emph{On $\cA\oplus\n$ consider the algebra structure
given by:}
\[ (a+\xi)(b+\eta)=ab+\xi b+a\eta\]
\emph{and} $\Psi:\cA\oplus\n\ni a+\xi \mapsto
\varphi(a)+f(\xi)\in\mathbb{C}$.

\emph{Let $(a_j,\xi_j)_j=1^\infty$ be a family from $\cA\oplus\n$
such that $\varphi(a_j)=f(\xi_j)=0$ and $(a_j,\xi_j)$ are type B
free in $\sys$.}

\emph{Then the limit in distribution of}

\[\frac{a_1+\xi_1+\dots a_N+\xi_N}{\sqrt{N}}\]
\emph{is the ''arcsine law''.}
 \end{cor}
 \end{remark}

\section{Poisson limit theorem}

 We will consider an analogue of the classical Bernoulli distribution in
a type $B$ probability space.

Let $A=(\alpha_1,\alpha_2)\in\mathbb{R}^2\subset \C$.  We call an
element $(a,\xi)\in\cA\times\n$ type $B$ Bernoulli with rate
$\Lambda$ and jump size $A$ if

\[E\left((a,\xi)^n\right)=\Lambda A^n\]
for some $\Lambda=(\lambda_1,\lambda_2)\in\C$

\begin{thm}
\emph{Let $\Lambda\in\C$ and $A\in\mathbb{R}^2$. Then the limit
distribution for $N\rightarrow\infty$ of the sum of $N$ free
independent type B Bernoulli variables with rate $\frac{\Lambda}{N}$
and jump size $A$}
\emph{has cumulants which are given by} $\kappa_n=\Lambda A^n$.
\end{thm}
\begin{proof}
We will introduce first several new notations in order to simplify
the writing. $\beta_N$ will stand for a type B Bernoulli variable
with rate $\frac{\Lambda}{N}$ and $s_N$ for a sum of $N$ such free
independent variables. $\mu$ will denote the Moebius function of the
lattice $NC(n)$ and, for $\pi\in NC(n)$ and $\beta\in \cA\times\n$,
we will use the notation
\[
M_{\pi}(\beta)=\prod_{B=\text{block of} \pi}M_{card(B)}(\beta)
\]
where $M_n(\beta)=E(\beta^n)$ is the $n$-th moment of $\beta$.

With the above notations, equation (\ref{mcum}) gives
\begin{eqnarray*}
\kappa_n(\beta_N)&=& \sum_{\pi\in NC(n)}
M_\pi(\beta_N)\mu(\pi,1_n)\\
&=&\frac{\Lambda}{N}A^n+
\sum_{
\begin{array}{cc}
\pi\in NC(n)\\
1_n\neq\pi\\
\end{array}
}
M_\pi(\beta_N)\mu(\pi,1_n)\\
&=&\frac{\Lambda}{N}A^n +O\left(\frac{1}{N^2}\right)
\end{eqnarray*}
Therefore
\[\lim_{N\rightarrow\infty}\kappa_n(s_N)=
\lim_{N\rightarrow\infty}N\kappa_n(\beta_N)=\Lambda A^n.\]
\end{proof}
Like in the type $A$ case, we have the following:
\begin{conseq}\emph{The square of a type $B$ random variable $(a,\xi)$ with distribution
given by the central limit theorem such that
$E\left((a,\xi)^2\right)=\sigma\in\C$ is a type $B$ free Poisson
element of rate $\sigma$ and jump size $(1,0)$.}
\end{conseq}

\begin{remark}\emph{The first component of the moments of a type $B$ free Poisson
variable coincides to the type $A$ case, therefore are given by a
probability measure on $\mathbb{R}$. In general, the second
component of the moments of a type $B$ free Poisson random variable
are not the moments of a real measure.}
\end{remark}

The first part of the assertion is clear. For the second part, we
will consider the particular case when $\lambda_2=0$ and
$\lambda_1=\lambda$ is close to 0 and $\alpha_1=\alpha_2=\alpha$. It
follows that
\[\kappa_n=\Lambda A^n=\left((\lambda,0)
(\alpha^n, n\alpha^{n})\right).\]

Since equation (\ref{mcum}) implies
\begin{eqnarray*}
M_2&=&\kappa_2+\kappa_1^2
=(\lambda+\lambda^2)A\\
M_3&=&\kappa_3+3\kappa_1\kappa_2+\kappa_1^3\\
&=&(\lambda+3\lambda^2+\lambda^3)A^3\\
 M_4&=&\kappa_4+4\kappa_1\kappa_3+
 2\kappa_2^2+6\kappa_2\kappa_1^2+\kappa_1^4\\
 &=&(\lambda+6\lambda^2+6\lambda^3+\lambda^4)A^4
\end{eqnarray*}
we have that the second components are given by:
\begin{eqnarray*}
\mathfrak{m}_2&=&2(\lambda+\lambda^2)\alpha^2\\
\mathfrak{m}_3&=&3(\lambda+3\lambda^2+\lambda^3)\alpha^3\\
\mathfrak{m}_4&=&4(\lambda+6\lambda^2+6\lambda^3+\lambda^4)\alpha^4\\
\end{eqnarray*}

A necessary condition for $\{\mathfrak{m}_k\}_{k\geq1}$ to be the
moments of a measure on $\mathbb{R}$ (see \cite{bsimon},
\cite{nisp}) is that
\[\mathfrak{m}_2\mathfrak{m}_4\geq \mathfrak{m}_3^2\]
It amounts to
\[8(\lambda+\lambda^2)(\lambda+6\lambda^2+6\lambda^3+\lambda^4)\alpha^6\geq
9(\lambda+3\lambda^2+\lambda^3)^2\alpha^6\] that is
\begin{eqnarray*}
8(1+\lambda)(1+6\lambda+6\lambda^2+\lambda^3)
&\geq&
9(1+3\lambda+\lambda)^2\\
8+O(\lambda)&\geq&9+O(\lambda)
\end{eqnarray*}
which, for $\lambda$ small enough, does not hold true.

\end{document}